\documentclass[12pt]{amsart}
\usepackage{times, amscd}

\begin{document}
\numberwithin{equation}{section}

\def\1#1{\overline{#1}}
\def\2#1{\widetilde{#1}}
\def\3#1{\widehat{#1}}
\def\4#1{\mathbb{#1}}
\def\5#1{\frak{#1}}
\def\6#1{{\mathcal{#1}}}

\newcommand{\de}{\partial}
\newcommand{\R}{\mathbb R}
\newcommand{\al}{\alpha}
\newcommand{\tr}{\widetilde{\rho}}
\newcommand{\tv}{\widetilde{\varphi}}
\newcommand{\tO}{\widetilde{\Omega}}
\newcommand{\hv}{\hat{\varphi}}
\newcommand{\tu}{\tilde{u}}
\newcommand{\tx}{\tilde{x}}
\newcommand{\ty}{\tilde{y}}
\newcommand{\tz}{\tilde{z}}
\newcommand{\usc}{{\sf usc}}
\newcommand{\tF}{\tilde{F}}
\newcommand{\tM}{\widetilde{M}}
\newcommand{\debar}{\overline{\de}}
\newcommand{\Z}{\mathbb Z}
\newcommand{\C}{\mathbb C}
\newcommand{\Po}{\mathbb P}
\newcommand{\zbar}{\overline{z}}
\newcommand{\G}{\mathcal{G}}
\newcommand{\So}{\mathcal{U}}
\newcommand{\Ko}{\mathcal{K}}
\newcommand{\U}{\mathcal{U}}
\newcommand{\B}{\mathbb B}
\newcommand{\oB}{\overline{\mathbb B}}
\newcommand{\Cur}{\mathcal D}
\newcommand{\Dis}{\mathcal Dis}
\newcommand{\Levi}{\mathcal L}
\newcommand{\SP}{\mathcal SP}
\newcommand{\Sp}{\mathcal Q}
\newcommand{\Ma}{\mathcal M}
\newcommand{\mF}{\mathcal F}
\newcommand{\mE}{\mathcal E}
\newcommand{\mR}{\mathcal R}
\newcommand{\mN}{\mathcal N}
\newcommand{\mI}{\mathcal I}
\newcommand{\tU}{\tilde{U}}
\newcommand{\mNtF}{\mathcal N_{\tmF}}
\newcommand{\tmF}{\widetilde{\mathcal F}}
\newcommand{\Co}{\mathcal C}
\newcommand{\Hol}{{\sf Hol}(\mathbb H, \mathbb C)}
\newcommand{\Aut}{{\sf Aut}(\mathbb D)}
\newcommand{\D}{\mathbb D}
\newcommand{\oD}{\overline{\mathbb D}}
\newcommand{\Ol}{\mathcal{O}}
\newcommand{\OM}{\mathcal{\widetilde{O}}_{\widetilde{M}}}
\newcommand{\loc}{L^1_{\rm{loc}}}
\newcommand{\loci}{L^\infty_{\rm{loc}}}
\newcommand{\la}{\langle}
\newcommand{\ra}{\rangle}
\newcommand{\thh}{\tilde{h}}
\newcommand{\N}{\mathbb N}
\newcommand{\kd}{\kappa_D}
\newcommand{\Hr}{\mathbb H}
\newcommand{\ps}{{\sf Psh}}
\newcommand{\tg}{\widetilde{\gamma}}

\newcommand{\subh}{{\sf subh}}
\newcommand{\harm}{{\sf harm}}
\newcommand{\ph}{{\sf Ph}}
\newcommand{\tl}{\tilde{\lambda}}
\newcommand{\ts}{\tilde{\sigma}}

\def\va{\varphi}
\def\Re{{\sf Re}\,}
\def\Im{{\sf Im}\,}

\def\dist{{\rm dist}}
\def\const{{\rm const}}
\def\rk{{\rm rank\,}}
\def\id{{\sf id}}
\def\aut{{\sf aut}}
\def\Aut{{\sf Aut}}
\def\CR{{\rm CR}}
\def\GL{{\sf GL}}
\def\U{{\sf U}}

\def\la{\langle}
\def\ra{\rangle}

\newtheorem{theorem}{Theorem}[section]
\newtheorem{lemma}[theorem]{Lemma}
\newtheorem{proposition}[theorem]{Proposition}
\newtheorem{corollary}[theorem]{Corollary}

\theoremstyle{definition}
\newtheorem{definition}[theorem]{Definition}
\newtheorem{example}[theorem]{Example}

\theoremstyle{remark}
\newtheorem{remark}[theorem]{Remark}
\numberwithin{equation}{section}

%tentative title
\title[Perturbation of Baum-Bott residues]{Perturbation of Baum-Bott residues}
\author[F. Bracci]{Filippo Bracci}
\author[T. Suwa]{Tatsuo Suwa}
\address{F. Bracci: Dipartimento Di Matematica\\
Universit\`{a} di Roma \textquotedblleft Tor Vergata\textquotedblright\ \\
Via Della Ricerca Scientifica 1, 00133 \\
Roma, Italy} \email{fbracci@mat.uniroma2.it}
\address{T. Suwa: Department of Mathematics\\
Hokkaido University\\
Sapporo 060-0810\\
Japan} \email{tsuwa@sci.hokudai.ac.jp}

%\subjclass[2000]{Primary 32H50, 32A10. Secondary 30D05.}


\thanks{Research partially supported by  a project PRIN2007 and a grant of JSPS}

\begin{abstract}
We prove that Baum-Bott residues vary continuously under smooth
deformations of holomorphic foliations. This provides an
effective way to compute residues.
\end{abstract}

\maketitle

\section{Introduction}

A holomorphic foliation $\mF$ on a complex manifold $M$ is
known to produce a ``holomorphic action'', as discovered by P.
Baum and R. Bott in \cite{BB2}, on the virtual bundle $TM/\mF$.
Such a partial holomorphic action provides a holomorphic
connection for the bundle $TM/\mF$ along $\mF$ outside the
singularities of $\mF$ and thus produces localization of
sufficiently high degree classes of $TM/\mF$ around the
singularities of $\mF$. Such localizations are called
``Baum-Bott residues'' (see \cite[Thm. 2]{BB2}, \cite[Ch.VI,
Thm. 3.7]{Su2}). When the singularity is isolated the Baum-Bott
residue can be expressed in terms of a Grothendieck residue
(see \cite[(0.6)]{BB2}). When the singular set is non-isolated
in some cases some formulas are available (see \cite[Thm.
3]{BB2} and \cite{D}) but, in general, explicit computation of
the residues is rather difficult.

The aim of the present paper is to study the behavior of the
Baum-Bott residues under smooth deformations. This provides an
effective tool for computing residues  explicitly.

More in details, we consider a smooth deformation of a complex
manifold. This is essentially a smooth fibration over a smooth
manifold, whose fibers are complex manifolds (see Section
\ref{def}). On each such a fiber we consider a holomorphic
foliation which varies smoothly  (see Section \ref{def-fol}).
We prove that the Baum-Bott residues (when taken together
suitably) vary continuously under smooth deformations.

We state here a simple consequence of our main Theorem
\ref{main} for the case of classes of top degree, referring the
reader to Section \ref{Bottres} for the general case. Thus, let
$P$ be a real manifold, the ``parameter space''. Let
$\tM:=\{M_t\}_{t\in P}$, be a deformation of complex manifolds
of dimension $n$. Let $\tmF:=\{\mF_t\}$ be a deformation of
holomorphic foliations on $M_t$. Then $\tmF$ defines naturally
a smooth foliation on $\tM$ (see Section \ref{def-fol}).

Suppose the singular set $S_{t_0}$ of $\mF_{t_0}$ in $M_{t_0}$
is compact and connected. The analytic set $S_{t_0}$ is
contained in a connected component in $\tM$ of the singular set
of the smooth foliation $\tmF$, and we denote by $S_t$ the
intersection of such component with $M_t$. The set $S_t$ is
contained in the singular set of $\mF_t$ but in general may not
be connected. Thus, we let $S_t=\cup S_t^\lambda$ be the
connected components decomposition of $S_t$.

\begin{theorem}\label{main-intro}
Suppose that $S_t$ is compact for all $t\in P$. Let $\varphi$
be a homogeneous symmetric polynomial of degree $n$ and denote
by $\hbox{BB}_\va(\mF_t;S^\lambda_t)$ the Baum-Bott residue of
$\mF_t$ at $S^\lambda_t$. Then
\[
\lim_{t\to t_0}\sum_\lambda \hbox{BB}_\va(\mF_t;S^\lambda_t)
=\hbox{BB}_\va(\mF_{t_0};S_{t_0}).
\]
\end{theorem}

A general version of the previous theorem is Theorem
\ref{main}, whose proof is contained in Sections \ref{Bottvan}
and \ref{Bottres}. The rough idea of the proof is to define a
special connection on the regular part of the bundle
$T\tM/\tmF$ such that on each fiber $M_t$ defines the special
connection given by the Baum-Bott action and see the residues
as the integral of a smooth form on $\tM$ along the fibers.

In Section \ref{examples} we give an explicit example of the
previous result. In particular, aside from explicit
computation, the example shows that if the residues in the same
connected component of $\tM$ are not taken together, continuity
is lost.

\medbreak

Part of this work was done while the first named author was
visiting the University of Tokyo. We would like to thank Prof.
J. Noguchi for providing us inspiring environment for research.

\section{Deformation of manifolds}\label{def}

The theory of deformation of complex structures was first
systematically developed by K. Kodaira and D. C. Spencer
\cite{KS}, here we recall the basic material relevant for our
needs.

\begin{definition}
A {\sl  deformation of manifolds} is a triple $(\tM, P, \pi)$,
where $P$ is a $C^\infty$ manifold of real dimension $s$,
called the {\sl parameter space},  $\tM$ is  a $C^\infty$
manifold of real dimension $2n+s$, called the {\sl ambient
manifold}, and $\pi: \tM\to P$ is a surjective $C^\infty$ map
such that there exists a covering $\{U_\al\}$ (called an {\sl
adapted deformation coordinates covering} of $\tM$) with the
following properties:
\begin{enumerate}
\item for each $\al$, the open set $U_\al$ is   diffeomorphic
to $D\times V$, where $D$ is an open set of $\C^n$ and $V$
is an open set of $\R^s$, with coordinates
$(z^\al_1,\ldots, z^\al_n, t_1^\al,\ldots, t_s^\al)$,
\item $\pi(U_\al)$ is diffeomorphic to $V$ and $\pi$ is
compatible with the projection $D\times V\to V$,
\item on $U_\al\cap U_\beta\neq \emptyset$ we may express as
\begin{equation}\label{cambio}
\begin{cases}
z^\beta_i=z^\beta_j(z^\al, t^\al)\quad i=1,\ldots, n\\
t^\beta_j=t^\beta_j(t^\al)\quad j=1,\ldots, s
\end{cases}
\end{equation}
and, for each fixed $t^\al$, the map $z^\al\mapsto
z^\beta(z^\al, t^\al)$ is holomorphic.
\end{enumerate}
\end{definition}

For $t\in P$ we let $M_t:=\pi^{-1}(t)$ be the fiber over $t$.
By definition the fibers $M_t$, for $t\in P$, are complex
manifolds. In particular we can define the sheaf $\OM$ of
$C^\infty$ functions on $\tM$ such that $f\in \OM(U)$ if for
all $x\in U$, $f|_{\pi^{-1}(\pi(x))}  \in
\Ol_{\pi^{-1}(\pi(x))}(U\cap {\pi^{-1}(\pi(x))})$.

\begin{remark}
Let $U_\al\subset \tM$ be a coordinate chart of an adapted
coordinate covering for $\tM$. A function $f$ belongs to
$\OM(U_\al)$ if and only if $f(z_\al, t_\al)$ is a $C^\infty$
function such that $f(\cdot, t_\al)$ is holomorphic (note that
this
 is well defined by \eqref{cambio}).
\end{remark}

\begin{definition}
Let $E$ be a $C^\infty$ complex vector bundle of rank $r$ over
$\tM$. We say that $E$ is an {\sl $\OM$-(vector) bundle} if
there exists a trivializing atlas $\{U_\al\}$ for $E$, with
frames $\{e_1^\al, \ldots, e_r^\al\}$ for $E|_{U_\al}$,  such
that the transition matrices with respect to those frames have
entries which are  local sections of $\OM$. Such frames
$\{e_1^\al, \ldots, e_r^\al\}$ are called {\sl $\OM$-frames}.
\end{definition}

Given an $\OM$-bundle $E$, we denote by $\OM(E)$ the
$\OM$-module of $\OM$ sections of $E$. Namely, $s\in \OM(E)(U)$
is a $C^\infty$ section of $E$ over the open set $U\subset \tM$
such that in any $\OM$-frame $\{e_1^\al, \ldots, e_r^\al\}$
over $U_\al$ with $U_\al\cap U\neq \emptyset$ the section $s$
is given by
\[
s(z^\al, t^\al)=\sum_{j=1}^r f^\al_j(z^\al, t^\al) e_j^\al, \quad f^\al_j \in \OM(U_\al\cap U).
\]

Let $T_{\R}\pi:=\ker \pi_\ast$. Since the fibers of the
fibration $\pi: \tM \to P$ are holomorphic, we can define the
complex vector bundles
\[
T\pi:=\bigcup_{x\in \tM} T_x \pi^{-1}(\pi(x)), \quad \overline{T\pi}:=\bigcup_{x\in \tM} \overline{T_x \pi^{-1}(\pi(x))}.
\]
Local frames for $T\pi$ and $\overline{T\pi}$ in an adapted
deformation coordinates covering are given respectively by
$\{\frac{\de}{\de z^\al_j}\}_{j=1,\ldots, n}$ and
$\{\frac{\de}{\de \overline{z}^\al_j}\}_{j=1,\ldots, n}$ and
\[
T_{\R}\pi\otimes \C=T\pi \oplus \overline{T\pi}.
\]
Using an adapted deformation coordinates covering, by
\eqref{cambio}, it is easy to see that $T\pi$ is an $\OM$-vector
bundle over $\tM$. Moreover, it has a natural structure of {\sl
$\OM$-Lie algebra}, namely, using local coordinates, one can
easily see that if $v, w \in \OM(T\pi)(U)$ then
\[
[v,w]\in \OM(T\pi)(U).
\]

\section{Deformation of foliations}\label{def-fol}

Deformations of holomorphic foliations, especially from the
point of view of moduli spaces, have been studied by a number
of authors ({\sl e.g.} \cite{G}, \cite{P}, \cite{R}). Here we
consider  $C^\infty$ families of singular holomorphic
foliations.

Let ${\mathcal S}$ be an $\OM$-module. We say that ${\mathcal
S}$ is {\sl coherent}   if, for each point $x\in \tM$ there
exists an open neighborhood $U\subset \tM$ and two integers
$p,q \geq 0$ such that
\begin{equation}\label{coer-S}
\OM|_U^p\stackrel{\va}{\longrightarrow} \OM|_U^q\longrightarrow {\mathcal S}|_U\to 0,
\end{equation}
is an exact sequence of $\OM|_U$-modules.

\begin{definition}
Let $(\tM, P, \pi)$ be a deformation of manifolds . A coherent
$\OM$-submodule $\tmF$ of $\OM(T\pi)$ such that $[\tmF,
\tmF]\subset \tmF$ is called a {\sl deformation of foliations}.
\end{definition}

Given a deformation of foliations $\tmF$ on a deformation of
manifolds $(\tM, P, \pi)$, we denote by ${\mathcal C}^\infty_P$
the sheaf of germs of complex valued smooth functions on $P$,
and for each $t\in P$, by $\mI_t:=\{f\in {\mathcal
C}^\infty_{P} : f(t)=0\}$ the ideal sheaf of smooth functions
vanishing at $t$. The set $\mR:=\pi^\ast{\mathcal
C}^\infty_{P}$  is the sheaf of smooth functions on $\tM$ that
are constant along the fibers, and it is naturally a subsheaf
of $\OM$. Noting that $\mR/\pi^*\mI_t$ is supported on
$M_t=\pi^{-1}(t)$, we define
\[
\mF_{t}:=\tmF\otimes_{\mR} \mR/\pi^\ast\mI_t.
\]
Note that $\OM\otimes_{\mR} \mR/\pi^\ast\mI_t=\Ol_{M_t}$, the sheaf of holomorphic
functions on $M_t$. Hence, if $\mathcal E$ is an $\OM$-module
over $\tM$, then $\mathcal E\otimes_{\mR} \mR/\pi^\ast\mI_t$ is an $\Ol_{M_t}$-module over $M_t$.

In particular, the sheaf $\mF_t$ is an $\Ol_{M_t}$-module. In
adapted deformation coordinates, if $X_1,\ldots, X_r$ are local
generators of $\tmF$, given by
\[
X_j(z^\al, t^\al)=\sum f_{ij}(z^\al, t^\al)\frac{\de}{\de z^\al_i},
\]
then $\mF_{t_0}$ is locally generated by the $X_j(z^\al,
t_0^\al)$'s. Namely it is generated by the vector fields
\begin{equation}\label{gen-F}
X_j(z^\al, t_0^\al)=\sum f_{ij}(z^\al,t_0^\al)\frac{\de}{\de z^\al_i}
\end{equation}
obtained by evaluating $f_{ij}(z^\al,t^\al)$ at $t=t_0$. From
this remark, it follows easily:

\begin{lemma}
For all $t\in P$, the sheaf $\mF_t$ defines a holomorphic
foliation on $M_t$.
\end{lemma}

We have the following exact sequence of $\OM$-modules on $\tM$:
\begin{equation}\label{ex-OM}
0\longrightarrow \tmF \longrightarrow \OM(T\pi)\longrightarrow \mNtF \longrightarrow 0.
\end{equation}
The {\sl singular set} of $\tmF$ is by definition
\[
S(\tmF):=\{x \in \tM: {\mNtF}_{,x} \hbox{\ is not\ } {\mathcal O}_{\tM,x}-\hbox{free}\}.
\]

\begin{lemma}\label{lemmaA}
For each point $x\in \tM$ there exists an open neighborhood
$U\subset \tM$ and two integers $p,q \geq 0$ such that
\begin{equation}\label{coer-NF}
\OM|_U^p\stackrel{\va}{\longrightarrow} \OM|_U^q\longrightarrow \mNtF|_U\to 0,
\end{equation}
is an exact sequence of $\OM|_U$-modules. Moreover,
\[
S(\tmF)|_U=\{x\in U : \rk \va_x \hbox{\ is not maximal}\}.
\]
\end{lemma}

\begin{proof}
Since $\tmF$  is $\OM$-coherent and $\OM(T\pi)$ is $\OM$-locally free,  from
(\ref{ex-OM}) it follows that $\mNtF$ is $\OM$-coherent as well,
so that \eqref{coer-NF} holds. The final statement follows from
standard commutative algebra and \eqref{coer-NF}.
\end{proof}

\begin{lemma}\label{lemmaB}
For each  $t\in P$ such that $M_t\not\subset S(\tmF)$ the
following sequence of $\Ol_{M_t}$-modules over $M_t$ is exact:
\begin{equation}\label{long-Ex}
0\to \tmF\otimes_{\mR} \mR/\pi^\ast \mI_t
\to \OM(T\pi)\otimes_{\mR} \mR/\pi^\ast \mI_t\to
\mNtF\otimes_{\mR} \mR/\pi^\ast \mI_t \to 0.
\end{equation}
\end{lemma}

\begin{proof} Since taking tensor products is right exact, it suffices to prove that the second map
from the left is injective.

It  is  true on the stalk over
each $x\in M_t$ such that $x\not\in S(\tmF)$, since $\mN_{{\mathcal F},x}$ is ${\mathcal{\widetilde{O}}_{{\widetilde{M}},x}}$-free.
We  note that according
to Lemma \ref{lemmaA}, $S(\tmF)|_{U\cap M_t}=\{x\in U\cap M_t :
\rk \va_x \hbox{\ is not maximal}\}$. Hence, for $t$ fixed,
these equations give rise to an analytic subset $S(\tmF)\cap
M_t$ of $M_t$, provided $M_t\not\subset S(\tmF)$. As a
consequence, $S(\tmF)\cap M_t$ is thin in $M_t$.
This shows that, since $\tmF$ is a subsheaf of $\OM(T\pi)$, it is also true on the stalk
over $x\in  S(\tmF)\cap M_t$.
\end{proof}

For each $t\in P$ we have the following exact sequence of
$\Ol_{M_t}$-modules:
\begin{equation}\label{exNF}
0 \longrightarrow \mF_t \longrightarrow \Ol_{M_t}(T M_t) \longrightarrow \mN_{\mF_t}\longrightarrow 0.
\end{equation}

\begin{definition}
Let $t\in P$. If $M_t\subset S(\tmF)$, we let
$S(\mF_t):=M_t$. Otherwise we let
\[
S(\mF_t):=\{x \in M_t : \mN_{\mF_t,x} \hbox{\ is not } \Ol_{M_t}-\hbox{free}\}.
\]
\end{definition}

\begin{proposition}\label{singu}
For all $t\in P$ it holds
\[
S(\mF_t)=S(\tmF)\cap M_t.
\]
\end{proposition}

\begin{proof}
If $M_t\subset S(\tmF)$ there is nothing to prove.

Thus, assume $M_t\not\subset S(\tmF)$. Since
$\OM(T\pi)\otimes_{\mR} \mR/\pi^\ast\mI_t=\Ol_{M_t}(TM_t)$, comparing
\eqref{long-Ex} and \eqref{exNF} we see that
\begin{equation}\label{ugua}
\mN_{\mF_t}=\mNtF\otimes_{\mR} \mR/\pi^\ast \mI_t,
\end{equation}
from which the statement follows at once.
\end{proof}

\section{Relative Bott vanishing for a deformation of
foliations}\label{Bottvan}

In this section we discuss a  Bott type vanishing theorem for
deformations of foliations. Thus, we let  $(\tM, P, \pi)$ be a
deformation of manifolds  and $\tmF$ a deformation of foliations on
$\tM$. In this section we assume
\[
S(\tmF)=\emptyset
\]
so that there exists an $\OM$-subbundle $\tilde F$ of $T\pi$ such that $\tmF=\OM(\tilde F)$.

We refer to \cite{BB2} for the notion of partial connections (see also \cite{ABST}, \cite{ABT}, \cite{Su2}).
As an example, given an $\OM$-bundle $E$ over $\tM$, we can define a ``relative
$\debar$-connection'' for $E$ along $\overline{T\pi}$ as
follows. We let
\[
\debar_E : C^\infty_{\tM}(E) \to C^\infty_{\tM}(\overline{T^\ast\pi}\otimes E),
\]
imposing that, given an $\OM$-frame
$\{\sigma_1^\al,\ldots,\sigma_r^\al\}$, and a $C^\infty$
section of $E$, $\sigma^\al:=\sum f_j^\al \sigma_j^\al$, it
holds
\[
\debar_E (\sigma^\al)=\sum_{j=1}^r\sum_{k=1}^n \frac{\de f_j^\al}{\de \overline{z}^\al_k}d\overline{z}^\al_k \otimes \sigma^\al_j.
\]
Since the transition matrices for $E$ with respect to
$\OM$-frames contains only entries in $\OM$, it is easy to see
that such a definition is well given and it is a partial
connection for $E$ along $\overline{T\pi}$.

\begin{definition}
Let $E$ be an $\OM$-bundle over $\tM$ and let $\mE$ be the sheaf
of its $\OM$-sections. A {\sl flat partial $\OM$-connection}
for $E$ along $\tmF$ is a $\C$-linear map
\[
\delta: \mE\to\tmF^\ast \otimes \mE
\]
with the properties that for all $X\in \tmF$, $f,g\in \OM$ and
$\sigma\in \mE$
\[
\delta_{(f X)}( g\sigma)=f\left(g\delta_X (\sigma)+ dg(X)\sigma\right)
\]
and
\[
\delta_X \circ \delta_Y=0, \quad \forall X,Y \in \tmF.
\]
\end{definition}

If $\delta$ is as above, it induces a ($C^\infty$) partial connection
\[
\delta : C^\infty_{\tM}(E) \to C^\infty_{\tM}( F^\ast\otimes E)
\]
such that, for $X\in \tmF$ and
$\sigma\in \mE$, we have $\delta_X(\sigma)\in \mE$.
Thus
\[
\delta\oplus\bar\partial_E : C^\infty_{\tM}(E) \to C^\infty_{\tM}( (F^\ast\oplus \overline{T^\ast\pi})\otimes E)
\]
is a partial connection. We say that a connection $\nabla:
C^\infty_{\tM}(E) \to
C^\infty_{\tM}((T^\ast\tM\otimes\C)\otimes E)$ extends
$\delta\oplus\bar\partial_E$ if
$\nabla_X=(\delta\oplus\bar\partial_E)_X$ for all sections $X$
of $F\oplus \overline{T\pi}$. Such a connection $\nabla$ always
exists (cf. \cite{BB2}).

We have the following ``relative Bott vanishing'' theorem for
actions of deformations of foliations:

\begin{theorem}\label{relative Bott vanishing}
let  $(\tM, P, \pi)$ be a deformation of manifolds  and $\tmF$ a
deformation of foliations on $\tM$ of rank $p$. Assume that
$S(\mF)=\emptyset$. Let $\mE$ be the sheaf of $\OM$-sections of
an $\OM$-bundle $E$ over $\tM$. Assume there exists a flat
partial $\OM$-connection $\delta$ for $\mE$ along $\tmF$. Then, for  any connection $\nabla$ for $E$ extending $\delta\oplus\bar\partial_E$, denoting by
$\iota_t: M_t \hookrightarrow \tM$ the natural embedding, it
follows
\[
\iota_t^\ast (\varphi(\nabla))=0,
\]
for all $t\in P$ and all symmetric homogeneous polynomials
$\varphi$ of degree $d>n-p$.
\end{theorem}

\begin{proof}
Let $\tF$ be the $\OM$-bundle whose associated sheaf of
sections is $\tmF$. Write
\[
T \tM\otimes \C= \tF \oplus F_1\oplus \overline{T\pi} \oplus \pi^\ast(TP),
\]
where $F_1$ is any $C^\infty$ complement of $\tF$ in $T\pi$.

Let $K$ be the curvature  of $\nabla$. Let $\{s_1, \ldots,
s_p\}$ be a local $\OM$-frame for $\tF$, and $\{\frac{\de}{\de
\overline{z}_1},\ldots, \frac{\de}{\de \overline{z}_n}\}$ the
natural frame for $\overline{T\pi}$ in adapted deformation
coordinates. Since $\tF$ is an $\OM$-subbundle of $T\pi$, we
can write $s_j=\sum_{k=1}^n a_k(z,t)\frac{\de}{\de z_k}$ for
$j=1,\ldots, p$ and $a_k\in \OM$.  Hence, $[s_j,\frac{\de}{\de
\overline{z}_k}]=0$ for $j=1,\ldots, p$ and $k=1,\ldots, n$.

Arguing similarly as in the proof of \cite[Prop. 3.27]{BB2}
(see also \cite[Thm. 6.1]{ABT}) since $\OM$-sections of $E$
generate as $C^\infty_{\tM}$-module the sheaf of
$C^\infty$-sections of $E$, one can see that
\[
K(s_j,s_k)=K(s_j,\frac{\de}{\de \overline{z}_h})=K(\frac{\de}{\de \overline{z}_h}, \frac{\de}{\de \overline{z}_l})=0
\]
for all $j,k=1,\ldots, p$ and $h,l=1,\ldots, n$. For instance,
given $\sigma$ an $\OM$-section of $E$, we have
\[
K(s_j,\frac{\de}{\de \overline{z}_h})(\sigma)=\nabla_{s_j}(\nabla_{\frac{\de}{\de \overline{z}_h}} \sigma)-
\nabla_{\frac{\de}{\de \overline{z}_h}}(\nabla_{s_j} \sigma)-\nabla_{[s_j,\frac{\de}{\de \overline{z}_h}]}\sigma=0,
\]
because $\nabla_{\frac{\de}{\de \overline{z}_h}}
\sigma=(\debar_E)_{\frac{\de}{\de \overline{z}_h}}\sigma=0$ by
definition, since $\sigma$ is an $\OM$-section; $\nabla_{s_j}
\sigma$ is another $\OM$-section of $E$, hence
$\nabla_{\frac{\de}{\de \overline{z}_h}}(\nabla_{s_j}
\sigma)=(\debar_E)_{\frac{\de}{\de
\overline{z}_h}}(\nabla_{s_j} \sigma)=0$ and
$[s_j,\frac{\de}{\de \overline{z}_h}]=0$.

As a consequence, the entries of the matrix representing $K$
are $2$-forms belonging to the ideal generated by a dual basis
of $F_1$ (which has dimension $n-p$) and by $dt_1,\ldots,
dt_s$, where these latter are a basis of $\pi^\ast(T^\ast P)$.
Therefore, if $\varphi$ has degree greater than $n-p$, it
follows that
\[
\varphi(\nabla)=\sum \omega_j \wedge dt_j,
\]
for some $2d-1$ forms $\omega_j$, hence,
$\iota^\ast(\varphi(\nabla))=0$.
\end{proof}

We recall that if $M$ is a complex manifold and $\mF$ is a
non-singular holomorphic foliation on $M$ then there exists a
natural holomorphic partial connection $\delta$ for the normal
bundle of the foliation $\mN_{\mF}$ along $\mF$ given by the so
called {\sl Baum-Bott action} (see \cite{BB2}, \cite{Su2}).
Such a connection is {\sl flat}, in the sense that $\delta\circ
\delta=0$. It is defined as follows:
\begin{equation}\label{Bott}
\delta_X(\sigma):=\rho([X,\tilde{\sigma}])
\end{equation}
where $\sigma\in \mN_{\mF}$ is a holomorphic section of the
normal bundle to the foliation, $\tilde{\sigma}\in \Ol_M(TM)$
is a holomorphic section of the tangent bundle to $M$ such that
$\rho(\tilde{\sigma})=\sigma$, where $\rho: \Ol_M(TM)\to
\mN_{\mF}$ is the natural projection, and $X\in \mF$.

We are going to show that a deformation of foliations gives
rise to a flat partial $\OM$-connection for $\mNtF$ along
$\tmF$ such that its ``restriction'' to each fiber $M_t$ is the
holomorphic flat partial connection for the normal bundle to
$\mF_t$ given by the Baum-Bott action:

\begin{proposition}\label{relative action}
Let  $(\tM, P, \pi)$ be a deformation of manifolds  and $\tmF$ a
deformation of foliations on $\tM$. Assume that
$S(\tmF)=\emptyset$. Then there exists a flat partial
$\OM$-connection $\tilde{\delta}$ for $\mNtF$ along $\tmF$.
Moreover, if $\iota_t: M_t\hookrightarrow \tM$ is the natural
embedding, then $\iota_t^\ast(\tilde{\delta})$ is the
holomorphic flat partial connection for $\mN_{\mF}$ along
$\mF_t$ given by the Baum-Bott action.
\end{proposition}

\begin{proof}
Let $\tilde{\rho}: \OM(T\pi)\to \mNtF$ be the natural
projection. For $X\in \tmF$ and $\sigma\in\mNtF$ we define
\begin{equation}\label{relative-Bott}
\tilde{\delta}_X(\sigma):=\tilde{\rho}([X,\tilde{\sigma}]),
\end{equation}
where $\tilde{\sigma}\in \OM(T\pi)$ is such that
$\tilde{\rho}(\tilde{\sigma})=\sigma$. Involutivity of $\tmF$
shows that $\tilde{\delta}$ is well-defined and  flatness
follows from the Jacobi identity, so that $\tilde{\delta}$ is a
partial $\OM$-connection  for $\mNtF$ along $\tmF$.

Comparing \eqref{relative-Bott} with \eqref{Bott}, it is easy
to see that $\iota^\ast_t(\tilde{\delta})$ is the  flat partial
$\Ol_{M_t}$-connection for $\mN_{\mF_t}$ along $\mF_t$ given by
the Baum-Bott action.
\end{proof}

In particular, Theorem \ref{relative Bott vanishing} and
Proposition \ref{relative action} imply the following:

\begin{corollary}\label{vBott}
Let  $(\tM, P, \pi)$ be a deformation of manifolds  and $\tmF$ a
deformation of foliations on $\tM$. Assume that
$S(\tmF)=\emptyset$. Then there exists a connection $\nabla$
for $\mNtF$ such that, denoting by $\iota_t: M_t
\hookrightarrow \tM$ the natural embedding, it follows
\[
\iota_t^\ast (\varphi(\nabla))=0,
\]
for all $t\in P$ and all symmetric homogeneous polynomials
$\varphi$ of degree $d>n-p$.
\end{corollary}

\section{Residues of Baum-Bott types on  deformations of manifolds}\label{Bottres}

In this section we assume $(\tM, P, \pi)$ is a deformations of
manifolds and $\tmF$ is a deformation of foliations on $\tM$.
We also assume that $\mNtF$ admits a {\sl $C^\infty$ locally
free resolution}, namely, there exists an  exact sequence of
${\mathcal C}^\infty_{\tM}$-modules:
\begin{equation}\label{fin-siz}
0\to \mE_q\to \cdots \to \mE_0 \to \mNtF\otimes_{\OM} {\mathcal C}^\infty_{\tM} \to 0,
\end{equation}
such that each $\mE_j$ is locally ${\mathcal
C}_{\tM}^\infty$-free.

\begin{remark}
If $\tmF$ is locally $\OM$-free then such a condition is
satisfied with $q=1$ and $\mE_1=\tmF\otimes_{\OM} {\mathcal
C}^\infty_{\tM} $, $\mE_0=\OM(T\pi)\otimes_{\OM} {\mathcal
C}^\infty_{\tM} $.
\end{remark}

Let $E_j$ be the vector bundle over $\tM$ whose sheaf of
$C^\infty$ sections is $\mE_j$. Then $\mNtF$ is a virtual bundle in the
$K$-group $K(\tM)$ and its total Chern class is defined as
\[
c(\mNtF)=\prod_{i=0}^q c(E_i)^{(-1)^i}.
\]
We briefly sketch here the theory we need, and refer the reader
to \cite[Section 4]{BB2}, \cite{L} and \cite[Ch.II,  8]{Su2}
for details.

Let $\tU_1$ be an open neighborhood of $S(\tmF)$ and let
$\tU_0:=\tM\setminus S(\tmF)$. We denote by $(\nabla_0^\bullet,
\nabla_1^\bullet)$ the family of $q+1$ connections compatible
with \eqref{fin-siz} and adapted to the covering ${\mathcal
\tU}:=\{\tU_0, \tU_1\}$ of $\tM$. Namely,
$\nabla_l^\bullet=(\nabla_l^q,\ldots, \nabla_l^0)$, $l=0,1$ is
a family
 such that $\nabla^j_l$ is a connection for
$E_j|_{\tU_l}$, $j=0,\ldots, q$, $l=0,1$ and the following
diagram is commutative for $i=1,\ldots, q$ and $l=0,1$:
\begin{equation}
\begin{CD}
 & E_i|_{\tU_l} @>\nabla_l^i>> C^\infty_{\tM}(T^\ast \tM \otimes E_i|_{\tU_l}) \\
 & @VVV  @VVV \\
 & E_{i-1}|_{\tU_l} @>\nabla_l^{i-1}>>
C^\infty_{\tM}(T^\ast \tM \otimes E_{i-1}|_{\tU_l})\\
\end{CD}
\end{equation}
Moreover, let $N_{\tF}$ be the vector bundle on $\tU_0$ whose
sheaf of sections is $\mNtF\otimes_{\OM} {\mathcal C}^\infty_{\tM} |_{\tU_0}$. Let $\nabla$ be an
extension of the  flat partial $\OM$-connection
$\tilde{\delta}$ for $\mNtF|_{\tU_0}$ along $\tmF$ given by
Proposition \ref{relative action}. It is then possible to
choose $\nabla_0^\bullet$ to be compatible with $\nabla$ (in
the sense explained before).

Now, we let $\varphi$ be a homogeneous symmetric polynomial of
degree $d>n-p$. One can define the class $\va(\mNtF )$ in the
\v{C}ech-de Rham cohomology \v{H}$^{2d}({\mathcal \tU})$ which
is represented by
\[
\va(\nabla_\ast^\bullet):=(\va(\nabla_0^\bullet),
\va(\nabla_1^\bullet), \va(\nabla_0^\bullet,
\nabla_1^\bullet)),
\]
where, by the compatibility condition,
$\va(\nabla_0^\bullet)=\va(\nabla)$ is a $2d$ form on $\tU_0$,
$\va(\nabla_1^\bullet)$ is the $2d$ form on $\tU_1$ associated
to the family $\nabla_1^\bullet$ and $\va(\nabla_0^\bullet,
\nabla_1^\bullet)$ is a $(2d-1)$-form on $\tU_0\cap \tU_1$ such
that $d \va(\nabla_0^\bullet,
\nabla_1^\bullet)=\va(\nabla_1^\bullet)-\va(\nabla_0^\bullet)$.
The \v{C}ech-de Rham cohomology \v{H}$^*({\mathcal \tU})$ is
naturally isomorphic to the de Rham cohomology $H^*_{dR}(\tM,
\C)$.

If $M_t\not\subset S(\tmF)$, tensorizing \eqref{fin-siz} with
$\mR/\pi^\ast \mI_t$ we obtain the following exact sequence of
${\mathcal C}^\infty_{M_t}$-modules (cf. the proof of Lemma
\ref{lemmaB}):
\begin{equation}\label{fin-sizF} 0\to \mE_q\otimes_{\mR} \mR/\pi^\ast \mI_t\to
\cdots \to \mE_0\otimes_{\mR} \mR/\pi^\ast \mI_t\to \mNtF\otimes_{\OM} {\mathcal C}^\infty_{\tM} \otimes_{\mR} \mR/\pi^\ast \mI_t \to 0,
\end{equation}
where $\mE_j\otimes_{\mR} \mR/\pi^\ast \mI_t$ is the sheaf of
$C^\infty$ sections of the restriction of the bundle $E_j$ to
$M_t$. By \eqref{ugua}, it is then easy to see the following:

\begin{lemma}
Let  $t\in P$ and let $\iota_t: M_t\to \tM$ be the natural
embedding. If $M_t\not\subset S(\tmF)$ then
$(\iota_t^\ast(\nabla_0^\bullet),\iota_t^\ast(\nabla_1^\bullet))$
is a family of connections for the virtual bundle $\mN_{\mF_t}$
compatible with \eqref{fin-sizF}.
\end{lemma}

By Corollary \ref{vBott} and by the compatibility condition, it
follows that for all homogeneous symmetric polynomials $\va$ of
degree $d>n-p$, the class $\va(\mN_{\mF_t})$ is represented in
the \v{C}ech-de~Rham cohomology associated to the covering
$\mathcal\tU\cap M_t$ of $M_t$ by the cocyle
\begin{equation*}
\begin{split}
 \va(\iota_t^\ast\nabla^\bullet_\ast)&=(\iota_t^\ast\va(\nabla_0^\bullet),
\iota_t^\ast\va(\nabla_1^\bullet), \iota_t^\ast\va(\nabla_0^\bullet,
\nabla_1^\bullet))=(\iota_t^\ast\va(\nabla),
\iota_t^\ast\va(\nabla_1^\bullet), \iota_t^\ast\va(\nabla_0^\bullet,
\nabla_1^\bullet))\\&=(0,
\iota_t^\ast\va(\nabla_1^\bullet), \iota_t^\ast\va(\nabla_0^\bullet,
\nabla_1^\bullet)).
\end{split}
\end{equation*}
Hence, since by Proposition \ref{singu}, $\tU_0\cap
M_t=M_t\setminus S(\mF_t)$, the previous cocycle defines a
localization of $\va(\mN_{\mF_t})$, call it $\va(\mN_{\mF_t},
\mF_t)$, in the relative \v{C}ech-de Rham cohomology
\v{H}$^{2d}(\mathcal\tU\cap M_t, M_t\setminus S(\mF_t))$. The
{\sl Baum-Bott residue}  is the image of $\va(\mN_{\mF_t},
\mF_t)$ by the Alexander homomorphism $A:
\hbox{\v{H}}^{2d}(\mathcal\tU\cap M_t, M_t\setminus
S(\mF_t))\to H_{dR}^{2n-2d}(\tU_1\cap M_t)^\ast$. If $S(\mF_t)$
is made of $k$ connected components than
$H_{dR}^{2n-2d}(\tU_1\cap M_t)^\ast$ is a direct sum of $k$
addends, and we can consider the Baum-Bott residue at each
connected component of $S(\mF_t)$.  If $\tU_1\cap M_t$ is a
regular neighborhood of $S(\mF_t)$ then the above residue can
be thought of as being in $H_{2n-2d}(S(\mF_t), \C)$.

Now, let  $S'(\tmF)\subseteq S(\tmF)$ be a connected component.
We assume that
\[
\hbox{$S_t:=M_t\cap S'(\tmF)$ is compact}\quad \forall t\in P.
\]
\begin{remark}
Note that even if $S(\tmF)$ is connected by assumption,
$S(\mF_t)$ might not.
\end{remark}
Let $\tilde{R}$ be a real manifold of dimension $2n+s$ with
boundary such that  $S'(\tmF)$ is contained in the interior of
$\tilde{R}$, no other components of $S(\tmF)$ intersect
$\tilde{R}$ and $\de \tilde{R}$ is transverse to $M_t$ for all
$t\in P$. Moreover, we can take $\tilde{R}$ in such a way that
$R_t:= \tilde{R}\cap M_t$ is compact for all $t\in P$.

 We let
$U_t:=\tU_1\cap M_t$ and denote by $H_{dR}^{\ast}(U_t)$ the de
Rham cohomology of $U_t$. By the previous construction, we can
express the {\sl Baum-Bott residue}
$\hbox{BB}_\va(\mF_t;S_t)\in H_{dR}^{2n-2d}(U_t)^\ast$ as
follows:
\begin{equation}\label{expre}
\hbox{BB}_\va(\mF_t;S_t): H_{dR}^{2n-2d}(U_t) \ni [\tau]\mapsto \int_{R_t} \iota_t^\ast\va(\nabla_1^\bullet)\wedge \tau -
\int_{\de R_t}\iota_t^\ast\va(\nabla_0^\bullet,\nabla_1^\bullet)\wedge \tau.
\end{equation}

\begin{remark}
Note that there exists a natural morphism
$H_{dR}^{2n-2d}(U_t)^\ast\rightarrow H_{dR}^{2n-2d}(M_t)^\ast$.
Therefore one can remove the dependence on $\tU_1$ in this
construction. Moreover, if $M_t$ is compact, then
$H_{dR}^{2n-2d}(M_t)^\ast=H_{2n-2d}(M_t)$.
\end{remark}

Now we are in good shape to prove our main result:

\begin{theorem}\label{main}
Let  $(\tM, P, \pi)$ be a deformation of manifolds  and $\tmF$
a deformation of foliations on $\tM$ of rank $p$. Suppose that
$\mNtF$ admits a  $C^\infty$ locally free resolution.  Let
$S'(\tmF)\subseteq S(\tmF)$ be a connected component of the
singular set of $\tmF$ and let $S_t:=M_t\cap S'(\tmF)$. Assume
that for all $t\in P$ the set $S_t$ is compact and $S_t\neq
M_t$. Let $\varphi$ be a homogeneous symmetric polynomial of
degree $d>n-p$. Under these assumptions, the Baum-Bott residue
$\hbox{BB}_\va(\mF_t;S_t)$ is continuous in $t\in P$. Namely,
for any $C^\infty$ $(2n-2d)$-form $\tilde{\tau}$ on $\tM$ such
that $\iota^\ast_t(\tilde{\tau})$ is closed for all $t\in P$,
\[
\lim_{t\to t_0}\hbox{BB}_\va(\mF_t;S_t)\left(\iota_t^\ast(\tilde{\tau})\right)
=\hbox{BB}_\va(\mF_{t_0};S_{t_0})\left(\iota_{t_0}^\ast(\tilde{\tau})\right).
\]
\end{theorem}

\begin{proof}
From the previous construction and \eqref{expre} it follows
that the Baum-Bott residues on $M_t$ are expressed by means of
smooth forms on $\tM$. Hence, they vary continuously.
\end{proof}

Note that, if $S_t$ is not connected and $S_t=\cup_\lambda
S_t^\lambda$ is its connected components decomposition, then
\[
\hbox{BB}_\va(\mF_t;S_t)=\sum_\lambda
\hbox{BB}_\va(\mF_t;S^\lambda_t).
\]

\section{An example}\label{examples}

In $\Po^3$ with homogeneous coordinates $[x_1:x_2:x_3:x_4]$ we
consider the vector field which is defined in the affine chart
$x_4\neq 0$  with coordinates $x=x_1/x_4, y=x_2/x_4, z=x_3/x_4$
by
\[
X(x,y,z):=x\frac{\de}{\de x}+x\frac{\de}{\de y}+y\frac{\de}{\de z}.
\]
The singularities are the line $L$ given by $x_1=x_2=0$ and the
point at infinity given by $Q:=[1:1:1:0]$ (see the next
expression \eqref{otherc}).

The vector field $X$ generates a one-dimensional foliation
$\mF$ given by $X: \Po^3\times \C \to T\Po^3$ on $\Po^3$. By
the Baum-Bott theorem, we can localize $\varphi(T\Po^3/\mF)$
for homogeneous symmetric polynomials $\varphi$ of degree $3$.
Such polynomials are essentially given by $c_1^3$, $c_1c_2$ and
$c_3$. Moreover, since $\mF$ is trivial, we obtain that
$\varphi(T\Po^3/\mF)=\varphi(T\Po^3)$. Let $\Ol(1)$ be the
hyperplane bundle on $\Po^3$ and let $\xi:=c_1(\Ol(1))\in
H^2_{dR}(\Po^3)$. From the Euler exact sequence, it follows
that $c(T\Po^3)=(1+\xi)^4$, from which
\begin{equation}\label{ctp}
\int_{\Po^3} c_1^3(T\Po^3)=64, \quad \int_{\Po^3} c_1c_2(T\Po^3)=24, \quad \int_{\Po^3} c_3(T\Po^3)=4.
\end{equation}
Changing coordinates, in the affine chart $x_3\neq 0$ with
coordinates $\tx=x_1/x_3, \ty=x_2/x_3, \tz=x_4/x_3$ the vector
field $X$ has the expression:
\begin{equation}\label{otherc}
X(\tx,\ty,\tz)=(\tx-\tx\ty)\frac{\de}{\de \tx}+(\tx-\ty^2)\frac{\de}{\de \ty}-\ty\tz \frac{\de}{\de \tz}.
\end{equation}
From this it follows that the first jet of $X$ at $Q$ is given
by the non-degenerate matrix
\[
A:=\left(
              \begin{array}{ccc}
                0 & -1 & 0 \\
                1 & -2 & 0 \\
                0 & 0 & -1 \\
              \end{array}
            \right).
\]
Hence since $Q$ is a non-degenerate isolated singularity for
$X$ it follows (see, {\sl e.g.} \cite[(0.7)]{BB2} or
\cite{Su2})
\begin{equation}\label{resi-nondeg}
\hbox{BB}_{\varphi}(X;Q)=\frac{\varphi(A)}{\det A},
\end{equation}
that is
\begin{equation}\label{expl-BB}
\hbox{BB}_{c_1^3}(X;Q)=27 \quad \hbox{BB}_{c_1c_2}(X;Q)=9 \quad \hbox{BB}_{c_3}(X;Q)=1.
\end{equation}
By the Baum-Bott theorem,
\[
\int_{\Po^3} \va(T\Po^3)=\hbox{BB}_{\varphi}(X;Q)+\hbox{BB}_{\varphi}(X;L).
\]
From this and by \eqref{ctp} and \eqref{expl-BB} we obtain
\begin{equation}\label{expl}
\hbox{BB}_{c_1^3}(X;L)=37 \quad \hbox{BB}_{c_1c_2}(X;L)=15 \quad \hbox{BB}_{c_3}(X;L)=3.
\end{equation}
However, computing such residues directly without using the
Baum-Bott theorem seem to be very complicated because the
singular set is not isolated.

We present now a deformation procedure which allows to compute
the previous residues and explain in practice how our Theorem
\ref{main-intro} works.

Let $\tM:=\Po^3\times (-1,1)$ and let  $\tmF$  be the
deformation of foliations defined by the vector fields $X_t$,
$t\in (-1,1)$, which on the chart $x_4\neq 0$ are defined as
\[
X_t(x,y,z)=(x+tz)\frac{\de}{\de x}+x\frac{\de}{\de y}+y\frac{\de}{\de z}.
\]
On the chart $x_3\neq 0$ the vector field $X_t$ is given by
\[
X(\tx,\ty,\tz)=(\tx-\tx\ty+t)\frac{\de}{\de \tx}+(\tx-\ty^2)\frac{\de}{\de \ty}-\ty\tz \frac{\de}{\de \tz}.
\]
The singularities of $X_t$ for $t\neq 0$ are given by
$O:=[0:0:0:1]$ and $P_j(t):=[u_{t,j}^2:u_{t,j}: 1:0]$ for
$j=1,2,3$, where the $u_{t,j}$'s are the three roots of the
equation $\lambda^3-\lambda^2-t=0$.

At the point $O$ the first jet of $X_t$, $t\neq 0$, is
non-degenerate and it is given by the matrix
\[
\left(
              \begin{array}{ccc}
                1 & 0 & t \\
                1 & 0 & 0 \\
                0 & 1 & 0 \\
              \end{array}
            \right).
\]
From this it and from \eqref{resi-nondeg},
\begin{equation}\label{XsO}
\hbox{BB}_{c_1^3}(X_t;O)=\frac{1}{t} \quad \hbox{BB}_{c_1c_2}(X_t;O)=0 \quad \hbox{BB}_{c_3}(X_t;O)=1.
\end{equation}

\begin{remark}
It is interesting to note that $\lim_{t\to
0}\hbox{BB}_{c_1^3}(X_t;O)=\infty$, namely, the residue by
itself is not continuous, but it is so when taken the sum of
the residues for the singularities which belong to the same
connected components in the ambient space $\tM$.
\end{remark}

At the point $P_j(t)$ the vector field $X_t$ has first jet
given by the matrix
\[
B(t,j):=\left(
              \begin{array}{ccc}
                1-u_{t,j} & -u_{t,j}^2 & 0 \\
                1 & -2u_{t,j} & 0 \\
                0 & 0 & -u_{t,j} \\
              \end{array}
            \right),
\]
with determinant $\det B(t,j)=  u_{t,j}^2(2-3 u_{t,j})$. Thus,
for $t\to 0$, $t\neq 0$ the points $P_j(t)$ are isolated
non-degenerate singularities for $X_t$ and one can use
\eqref{resi-nondeg} to compute the residues:
\begin{equation}\label{XsP}
\begin{split}
&\hbox{BB}_{c_1^3}(X_t;P_j(t))=\frac{(1-4u_{t,j})^3}{u_{t,j}^2(2-3 u_{t,j})} \quad \hbox{BB}_{c_1c_2}(X_t;P_j(t))=\frac{3(6u_{t,j}-1-8u_{t,j})}{u_{t,j}^2(2-3 u_{t,j})} \\& \hbox{BB}_{c_3}(X_t;P_j(t))=1.
\end{split}
\end{equation}

Now, as $t\to 0$, it follows that two of the roots of of the
equation $\lambda^3-\lambda^2-t=0$ tend to $0$ and one tends to
$1$. We assume that $u_{t,1}, u_{t,2}\to 0$ and $u_{t,3}\to1$.
Hence, if $S'(\tmF)$ is the connected component which contains
the line $L$ in the manifold deformation $M\times (-1,1)$, the
intersection of $S'(\tmF)$ with $M\times \{t\}$ is given by the
points $O, P_1(t), P_2(t)$. While, the connected component in
$M\times (-1,1)$ which contains $Q$ contains all the points
$P_3(t)$.

A direct computation---taking into account that
$u_{t,1}+u_{t,2}+u_{t,3}=1$,
$u_{t,1}u_{t,2}+u_{t,1}u_{t,3}+u_{t,2}u_{t,3}=0$ and
$u_{t,1}u_{t,2}u_{t,3}=t$---shows that for $\varphi=c_1^3,
c_1c_2, c_3$
\begin{equation*}
\begin{split}
&\lim_{t\to 0}\hbox{BB}_{\varphi}(X_t;P_3(t))=\hbox{BB}_{\varphi}(X;Q),\\
&\lim_{t\to 0}[\hbox{BB}_{\varphi}(X_t;P_1(t))+\hbox{BB}_{\varphi}(X_t;P_2(t))+\hbox{BB}_{\varphi}(X_t;O)]=\hbox{BB}_{\varphi}(X;L).
\end{split}
\end{equation*}

\end{document}